\documentclass[final,onefignum,onetabnum]{siamart220329}
\usepackage[margin=1.33in]{geometry}

\usepackage[T1]{fontenc}
\usepackage[utf8]{inputenc}
\usepackage{amssymb}
\usepackage{amsmath}
\usepackage{graphicx}
\usepackage{mathtools,xfrac}
\usepackage{mathrsfs}
\usepackage{dsfont}
\usepackage{autonum}
\usepackage{cite}

\usepackage{xcolor}

\usepackage{booktabs}
\usepackage{siunitx}

\allowdisplaybreaks

\newsiamremark{example}{Example}
\crefname{hypothesis}{Hypothesis}{Hypotheses}
\newsiamthm{fact}{Fact}
\crefname{fact}{Fact}{Facts}

\DeclareMathOperator{\dist}{dist} 

\DeclareMathOperator{\conv}{conv} 
\DeclareMathOperator{\cl}{cl} 

\DeclareMathOperator{\circum}{circ} 

\newcommand{\norm}[1]{\left \lVert #1 \right \rVert}

\newcommand{\closure}[1]{\cl\left({ #1 }\right)}

\def\re{\mathds R}
\def\na{\mathds N}

\def\bSS{\mathbf S}
\def\DD{\mathbf D}

\def\CC{\mathbf C}

\def\vx{\mathbf x}
\def\vy{\mathbf y}
\def\vz{\mathbf z}

\def\lV{\left\lVert }
\def\rV{\right\rVert }
\def\lv{\left\lvert }
\def\rv{\right\rvert}

\DeclareMathOperator{\Id}{Id}

\DeclareMathOperator{\inte}{int}

\newcommand{\CRMOp}{{\mathscr{C}}}

\makeatletter
\newcommand{\specificthanks}[1]{\@fnsymbol{#1}}
\makeatother
\usepackage[shortlabels]{enumitem}
\newlist{lista}{enumerate}{1}
\setlist[lista]{label=\alph*., nosep,leftmargin=*,align=right}

\newlist{listi}{enumerate}{1}
\setlist[listi]{label={\upshape(\roman*\upshape)},leftmargin=*,align=right, widest=iii,nosep, format=\bf}

\begin{document}

\title{A finitely convergent circumcenter method for the Convex Feasibility Problem\thanks{
    \funding{\textbf{RB} was partially supported by Brazilian agency Fundação de Amparo à Pesquisa do Estado do Rio de Janeiro (Grant E-26/201.345/2021); \textbf{YBC} was partially supported by the USA agency National Science Foundation (Grant DMS-2307328). \textbf{LRS} was partially supported by Brazilian agencies Conselho Nacional de Desenvolvimento Científico e Tecnológico (Grant 310571/2023-5) and Fundação de Amparo à Pesquisa do Estado de São Paulo (Grant 2023/03969-4). \textbf{DL} thanks the Brazilian agency Coordenação de Aperfeiçoamento de Pessoal de Nível Superior for the doctoral scholarship.}}}

\author{Roger Behling\textsuperscript{\specificthanks{4},}\thanks{{Department of Mathematics}, {Federal University of Santa Catarina}, {{Blumenau}-{SC} -- {89065-300}, {Brazil}}. (\email{rogerbehling@gmail.com}, \email{l.r.santos@ufsc.br})} 
\and Yunier Bello-Cruz\thanks{{{Department of Mathematical Sciences}, {Northern Illinois University}, {{DeKalb}-{IL} -- {60115-2828}, {USA}} (\email{yunierbello@niu.edu})}}
        \and Alfredo N.~Iusem\thanks{{School of Applied Mathematics}, {Fundação Getúlio Vargas}, {{Rio de Janeiro}-{RJ} -- {22250-900}, {Brazil}} (\email{alfredo.iusem@fgv.br})}
        \and Di Liu\thanks{{Instituto de Matemática Pura e Aplicada}, {{Rio de Janeiro}-{RJ} -- {22460-320}, {Brazil}} (\email{di.liu@impa.br})}
        \and Luiz-Rafael Santos\footnotemark[2]}

\headers{A finitely convergent 
circumcenter method for the CFP}{Behling, Bello-Cruz, Iusem, Liu and Santos}

        \maketitle

\begin{abstract}

In this paper, we present a variant of the circumcenter method for the Convex Feasibility Problem (CFP), ensuring finite convergence under a Slater assumption. The method replaces exact projections onto the convex sets with projections onto separating halfspaces, perturbed by positive exogenous parameters that decrease to zero along the iterations. If the perturbation parameters decrease slowly enough, such as the terms of a diverging series, finite convergence is achieved. To the best of our knowledge, this is the first circumcenter method for CFP that guarantees finite convergence.

\end{abstract}

\begin{keywords}
Convex Feasibility Problem, Finite convergence,
Circumcentered-reflection method, Projection methods.
\end{keywords}

\begin{MSCcodes} 49M27, 65K05, 65B99, 90C25
\end{MSCcodes}

\section{Introduction}\label{s1}
The convex feasibility problem (CFP) aims at finding a point in the intersection of $m$ closed and convex sets 
$C_i \subset \re^n$, $i=1,\ldots, m$, \emph{i.e.},
\[
\label{eq.CFP_1}
    \text{find } x^\star \in  C\coloneqq \bigcap_{i=1}^m C_i.
\]
Convex feasibility represents a modeling paradigm for solving numerous engineering and physics problems, such as image recovery \cite{Combettes:1996}, wireless sensor networks localization \cite{Hu:2016}, gene regulatory network inference \cite{Wang:2017}, and many others.

Projection-reflection based methods are widely recognized and effective schemes for solving a diverse range of feasibility problems, including \cref{eq.CFP_1}. This sort of methods continues to gain popularity due to their ability to strike a balance between high performance and simplicity, as evidenced by their extensive utilization (see, \emph{e.g.}, \cite{Bauschke:1996}). Among these methods, two particularly renowned and widely adopted methods are the classical Douglas-Rachford method (DRM) and its modifications (see, \emph{e.g.}, \cite{Bauschke:2014b,Douglas:1956,AragonArtacho:2018a}), and the famous method of alternating projections (MAP) (see, \emph{e.g.}, \cite{Bauschke:1993,Bauschke:2016}). 
The elementary Euclidean concept of \emph{circumcenter} has recently been employed to improve the convergence of those classical projection-reflection methods for solving the CFP \cref{eq.CFP_1}.

The circumcentered-reflection method (CRM) was first presented in \cite{Behling:2018} 
as an acceleration technique for DRM for the two set affine CFP. Since then, CRM has been shown as a valid and powerful new tool for solving (non)convex structured feasibility problems because of its ability to minimize the inherent \emph{zigzag} behavior of projection-reflection based methods, in particular. In~\cite{Behling:2020}, for instance, CRM was connected to the classical MAP. Moreover, CRM  obviates the  \emph{spiraling} behavior for the classical DRM \cite{Dizon:2022,Lindstrom:2022,Dizon:2022a}. There are already a plethora of works  where  circumtencered-based schemes were  studied; see for instance \cite{Behling:2024b, Araujo:2022, Arefidamghani:2021, Arefidamghani:2023, Behling:2020, Behling:2021b, Behling:2024, Behling:2023, Bauschke:2018, Bauschke:2020, Bauschke:2021b, Bauschke:2021d, Bauschke:2022a, Ouyang:2018, Ouyang:2021a, Ouyang:2022a, Ouyang:2022b, Ouyang:2023,Araujo:2021, Arefidamghani:2022a,Behling:2018a}.

The circumcenter of three points $x,y,z\in \re^n$, denoted as $\circum(x,y,z)$, is the point in $\re^n$ that lies in the affine manifold spanned by $x,y$ and $z$ and is equidistant to these three points. Given two closed convex sets $A,B\subset \re^n$, the CRM iterates by means of the operator $\CRMOp_{A,B}:\re^n\to \re^n$ defined as 
\[
    \label{eq.definition_Circumcenter}
    \CRMOp_{A,B}(x)\coloneqq\circum(x,R_A(x),R_B(R_A(x)),
\]
where
$P_A:\re^n\to A, P_B:\re^n\to B$ are the orthogonal projections onto $A,B$ respectively, $R_A = 2P_A-\Id$, $R_B = 2P_B-\Id$
(\emph{i.e.}, $R_A, R_B$ are the \emph{reflections} onto $A$, $B$, respectively), 
and $\Id$ is the identity operator in $\re^n$. Hence, the sequence generated by CRM is defined as 
\[\label{eq:def.CRM_operator}
x^{k+1}=\CRMOp_{A,B}(x^k).
\] 
If $x^k\in A\cap B$, then the sequence stops at iteration $k$, in which case, we say that the algorithm has finite convergence.

One limitation of CRM is that its convergence theory requires one of the sets to be an affine manifold. In \cite{Behling:2018}, it was pointed out that the iteration defined in \cref{eq:def.CRM_operator} may fail to be well-defined or to approach $A\cap B$, if such an assumption is not met. Subsequently, a specific counterexample was presented in \cite{AragonArtacho:2020} where CRM does not converge for two general convex sets. However, CRM can be used to solve the general CFP with $m$ general arbitrary closed convex sets by employing Pierra's product space reformulation, as presented in \cite{Pierra:1984}. The product space reformulation relies on the Cartesian product    $\CC\coloneqq C_1 \times \cdots \times C_m \subset \re^{nm}$ and on  $\DD\coloneqq \{(x,x,\ldots,x) \in \re^{nm} \mid x\in \re^n\}$. $\DD$ is said to be the \emph{diagonal subspace} in $\re^{nm}$. One can easily see that
\[
    \label{eq.Product_Space1}
    x^\star\in C \Leftrightarrow \vz^\star\coloneqq (x^\star,x^\star,\ldots,x^\star)\in \CC \cap \DD,
\]
where $C=\bigcap_{i=1}^m C_i$. Due to \cref{eq.Product_Space1}, solving \cref{eq.CFP_1} corresponds to solving
\[
    \label{eq.Product_Space2}
    \text{find }\vz^\star\in\CC \cap \DD.
\]
Since $\DD$ is an affine manifold, CRM can be applied to finding a point in $\CC \cap\DD$ in the product space $\re^{nm}$.

In \cite{Arefidamghani:2021}, it was proved that the Circumcenter Reflected Method (CRM) with Pierra's product space reformulation achieves a superior convergence rate compared to MAP. The iteration operator of MAP is simply the composition of two orthogonal projections $P_A$ and $P_B$, that is, $P_A P_B$. Furthermore, in the same paper, it was shown that for certain special cases, circumcenter schemes such as CRM attain a superlinear rate of convergence \cite{Arefidamghani:2021, Behling:2024}, and even linear convergence in the absence of an error bound. Notably, no other known method utilizing individual projections achieves such convergence rate even for these particular cases. 

More recently, an extension of CRM, called cCRM (acronym for centralized circumcentered-reflection method), was introduced in \cite{Behling:2024b} to overcome the drawback of CRM, namely the requirement that one of the sets be an affine manifold. The cCRM can solve the CFP for any pair of closed and convex sets and converges linearly under an error-bound condition (similar to a transversality hypothesis), and superlinearly if the boundaries of the convex sets are smooth. Yet, for solving the CFP with $m$ sets, it is necessary to go through the Pierra reformulation in the product space.

For solving the CFP with $m$ sets more efficiently, a successive extension of cCRM, called ScCRM, was developed in \cite{Behling:2024}. ScCRM avoids the product space reformulation and inherits from cCRM the linear and superlinear convergence rates under the error-bound and smoothness hypothesis, respectively.

All the aforementioned methods require several exact orthogonal projections onto some convex sets in each iteration. However, determining the exact orthogonal projection is computationally expensive in practice, except in some very special cases. This limitation was overcome in \cite{Araujo:2022}, where an approximate version of CRM, called CARM, is proposed. CARM replaces the exact projections onto the original sets with projections onto sets containing them ({\emph{e.g.}},  halfspaces or Cartesian products of halfspaces), which are easily computable.

In this paper, we improve upon the aforementioned methods by introducing an algorithm in the product space that  uses projections onto perturbed halfspaces, referred to as \emph{Perturbed Approximate Circumcenter Algorithm} (PACA).  PACA's  iteration reads as follows
\[
    \label{eq.def.PACA_iteration}
    \vx^{k+1} =  \CRMOp_{\bSS^k ,\DD}(\vx^k) = \circum(\vx^k,R_{\bSS^k}(\vx^k),R_{\DD}(R_{\bSS^k}(\vx^k))),
\]
where $\bSS^k$ is a suitable perturbed separating set for $\vx^k$.
For any $\vx^0 \in \DD$, we will be able to obtain a circumcenter sequence in $\DD$ generated  by means of  \cref{eq.def.PACA_iteration}  converging in finitely many steps to a point in $\CC \cap \DD$, if the Slater condition holds, meaning that the intersection of the convex sets has nonempty interior. 
This is possible by strategically  building the perturbed separating  set  
$\bSS^k$ (see \cref{eq.def.s_i_k,eq.def.bf_S_k}). These perturbed separating sets resemble the surrogate 
halfspaces used by \cite{Kiwiel:1995} in the context of CFP, and by \cite{Kiwiel:1997,Combettes:2000} in a more general framework. 

Finite convergence results for projection-based methods  with perturbations applied to CFP, under the Slater condition, have been obtained before \cite{Bauschke:2015b,Censor:2011,Crombez:2004,Kolobov:2021,Kolobov:2022,Polyak:2001,Iusem:1986a,DePierro:1988a};  we point out the readers to \cite[Tab.~1]{Kolobov:2022}, a summary where many of these results are compared. We mention that in  \cite{Bauschke:2017c}, it is proved that even the Douglas-Rachford method applied to the two-set setting can converge in finitely many steps under several assumptions regarding the underlying sets. 
Moreover, a generalized version of MAP  is proposed and analyzed in \cite{Falt:2017, Falt:2023}, and the algorithm can achieve finite convergence under the assumption of some regularity of the underlying convex sets at a solution~\cite[Thm.~6.1]{Falt:2023}.

The paper is organized as follows: In \cref{s2}, we present some definitions and preliminary material. In \cref{s3}, we state our algorithm, PACA,  \cref{s4}, shows the development of  the convergence analysis. Finally, \cref{s5} shows the linear convergence rate of the sequence generated by our algorithm. Finite convergence is proved in \cref{s6} under additional assumptions on the perturbation parameters. Finally, \cref{s7} presents numerical experiments for solving ellipsoids intersection problems, comparing PACA with CARM (in the product space), the Simultaneous (Cimmino~\cite{Cimmino:1938}) subgradient projection method (SSPM)  by \cite{Iusem:1986a}, and the Modified Cyclic (Kaczmarz~\cite{Kaczmarz:1937}) subgradient projection method  (CSPM) from~\cite{DePierro:1988a}.

\section{Preliminary results}\label{s2}
In this paper, we consider the CFP as in \cref{eq.CFP_1}, assuming that the sets $C_i$ are given as:
\[
    \label{eq.def.CFP_functionValues}
    C_i=\{x\in \re^n \mid f_i(x) \leq 0\},
\]
where $f_i: \re^n \rightarrow \re $ is a convex function for $i=1,\ldots,m$.

This formulation does not impose any limitations in principle, as we can always choose $f_i(x)=\left \Vert P_{C_i}(x)-x\right\Vert$. However, this choice is not consistent with the later imposition of the Slater hypothesis. Nevertheless, in most instances of the CFP, it is possible to formulate them as in \cref{eq.def.CFP_functionValues} using suitable functions $f_i$.

A basic assumption for our finite convergence result is the Slater condition, defined as follows:

\begin{definition}[Slater condition]
\label{def.Slater_condition}
    There exists $\hat x\in\re^n$ such that
    \[
    \label{eq.def.Slater_condition}
        f_i(\hat x) <0,
    \] for all $i=1,\ldots,m$.
\end{definition}

We continue by recalling the explicit formula of the orthogonal projection onto a halfspace.

\begin{proposition}
\label{lem.projection_to_halfspace}
    If $H=\{y\in \re^n\mid a^\top y\leq \alpha\}$, with $a\in \re^n$, $\alpha \in R$, then the orthogonal projection onto $H$ is given by
    \[
        \label{eq.def.projection_to_halfspace}
        P_H(x) = x- \frac{\max\{0,\alpha -a^\top x\}}{\|a\|^2}a.
    \]
\end{proposition}
\begin{proof}
The result follows from the definition of orthogonal projection after an elementary calculation.   
\end{proof}

The following proposition is a key result concerning the circumcenter operator $\CRMOp_{A,B}$, as defined in \cref{eq.definition_Circumcenter}. It represents the only result pertaining to circumcenter steps that will be used in our convergence analysis.

\begin{proposition}
\label{lem.CRM} Let $A$ and $B$ be closed convex subsets of $\re^n$ with nonempty intersection.
Consider the circumcenter operator $\CRMOp_{A,B}$ defined in
\cref{eq.definition_Circumcenter}. Suppose that $B$ is an affine manifold. Then, for all $x\in B$ there exists a closed and convex set $H(x)\subset\re^n$ such that,
    \begin{listi}
    \item $A\subset H(x)$;
        \item $\CRMOp_{A,B}(x)=P_{H(x)\cap B}(x)$;
        \item $\lV \CRMOp_{A,B}(x) -s \rV^2\le\lV x-s\rV^2 -\lV \CRMOp_{A,B}(x) - x\rV^2$ for all $s\in A$.
    \end{listi}
\end{proposition}
\begin{proof} The result follows from Lemma 3 in \cite{Behling:2018};  
cf. also Lemma 3.3 and Proposition 3.4 in \cite{Araujo:2022}.   
In fact $H(x)$ is a halfspace containing $A$, with an explicit formula, which ensures that item (i) holds, but this is irrelevant for our purposes. We also mention that item (ii) is the essential result, while item (iii) follows immediately from it.     
\end{proof}
 
A useful concept for analyzing the convergence of projection-based algorithms is the notion of \emph{Fej\'er monotonicity}. A sequence $\{y^k\}\subset\re^n$ is said to be \emph{Fej\'er monotone} with respect to a set $M\subset\re^n$ if $\lVert y^{k+1}-y\rVert\leq\lVert y^k-y\rVert$ for all $y\in M$ and all $k\in\mathbb{N}$. It is known that if $\{y^k\}$ is Fej\'er monotone with respect to $M$, then $\{y^k\}$ is bounded, and if it has a cluster point $y^\star\in M$, then the entire sequence $\{y^k\}$ converges to $y^\star$~\cite[Thm.~2.16]{Bauschke:1996}. In our analysis, we require a slightly weaker notion and also need to handle cases where $M$ is open, and the cluster points of $\{y^k\}$ belong to the boundary of $M$.

We now introduce the appropriate tools to address this situation.

\begin{definition}[Fejér* monotonicity]
\label{def.Fejer*monotone}
    Let $M\subset \re^n$ and consider a sequence $\{y^k\} \subset \re^n$.  We say that $\{y^k\}$ is \emph{Fej\'er* monotone with respect to $M$} if for any point $y \in M$, there exists $N(y)\in\na$ such that
    \[\label{eq:def.Fejer*monotone}
        \lV y^{k+1}-y\rV\le\lV y^k -y\rV,
    \]
    for all $k\ge N(y)$.
\end{definition}

The main difference with the usual Fej\'er monotonicity
notion lies in the fact that now the decreasing distance property holds for the tail of the sequence, starting at some index which depends on the considered point $y\in M$. We present next some useful properties of Fej\'er* monotonicity.

\begin{proposition}[Characterization of Fejér* monotonicity]
    \label{lem.Fejer*monotone.char} 
    Let $\left\{y^k\right\}\subset\re^n$ be a Fejér* monotone  w.r.t.~a nonempty set $M$ in $\re^n$. Then,

     \begin{listi}
        
\item $\left\{y^k\right\}$ is bounded;
\item for every $y\in M$, the scalar sequence $\left\{\norm{y^k-y}\right\}$ converges;
\item  $\left\{y^k\right\}$ is Fejér* monotone w.r.t.~$\conv(M)$;
\item for every $\Bar{y} \in \closure{\conv (M)}$, the closure of $\conv (M)$, the scalar sequence $\left\{\left\|y^k-\Bar{y}\right\|\right\}$ converges.

\end{listi}
\end{proposition}
  
\begin{proof}
    For proving \textbf{(i)}, take any point $y\in M$. From the definition of Fej\'er* monotonicity, we conclude that $y^k$ belongs to the ball with center at $y$ and radius $\lV y^{N(y)}-y\rV$ for all $k\ge N(y)$. Consequently, $\{y^k\}$ is bounded.
   
    Item \textbf{(ii)} is a direct consequence of   \cref{eq:def.Fejer*monotone}, since the $N(y)$-tail of sequence $\left\{\norm{y^k-y}\right\}$  is monotone and bounded, so the sequence converges.

    For item \textbf{(iii)}, take any $y\in \conv(M)$. Thus, $y$ can be written as  $y= \sum_{i=1}^p \lambda_i y_i$, where $y_i\in M$,  $\lambda_i \in [0,1]$ for all $i = 1,2,\ldots,p$, $\sum_{i=1}^p \lambda_i =1$ and $p\in \na$. Taking into account that  $\{y^k\}$ is Fejér* monotone w.r.t.~$M$, for each $i=1, 2,\ldots,p$, there exists $N(y_i)$ such that $\lV y^{k+1}-y_i\rV\le\lV y^k -y_i\rV$ for all $k\ge N(y_i)$. Squaring both sides of the last inequality and rearranging the terms, we get
    \begin{align}   \label{eq:Fejer*monotone_foo}
        0 &\leq \lV{y^k}\rV^2 - \lV{y^{k+1}}\rV^2  - 2 \left(y^{k}-y^{k+1}\right)^\top y_i.
    \end{align}
The above inequality also holds for all $k\ge K \coloneqq \max\{N(y_1), N(y_2), \ldots, N(y_p)\}$. Multiplying \cref{eq:Fejer*monotone_foo}, for each $i = 1,2,\ldots,p$, by the respective $\lambda_i$ and adding up, we obtain
    \begin{align}   
        0 
        &\le \lV{y^k}\rV^2 -\lV{y^{k+1}}\rV^2  - 2 \left(y^{k}-y^{k+1}\right)^\top \left(\sum_{i=1}^p \lambda_i y_i\right) \\ 
        &= \lV{y^k}\rV^2 - \lV{y^{k+1}}\rV^2  - 2 \left(y^{k}-y^{k+1}\right)^\top y,
    \end{align}
    and therefore, $\lV y^{k+1}-y\rV\le\lV y^k - y\rV$, for all $k\ge K$. So $\{y^k\}$ is Fejér* monotone w.r.t.~$\conv(M)$.

    Regarding item \textbf{(iv)}, for any $w\in \conv(M)$, since item \textbf{(iii)} implies that  $\{y^k\}$ is Fejér monotone w.r.t.~$\conv(M)$,   item \textbf{(ii)} yields that the scalar sequence $\{\|y^k-w\|\}$ converges. Take now $\Bar{y} \in \closure{\conv(M)}$, and a sequence $\left\{w^\ell\right\}\subset \conv(M)$ such that
    $w^\ell \rightarrow \Bar{y}$.

Using the triangle inequality, we have, for all $\ell \in \na$,
\begin{align}
  -   \left\|w^\ell-\Bar{y}\right\| \leq 
  \left\|y^k -\Bar{y}\right\| - \left\|y^k -w^\ell\right\|  \leq \left\|w^\ell-\Bar{y}\right\|.
\end{align}
Taking $\liminf$ and $\limsup$ with respect to $k$ in the above inequalities, we obtain
\begin{align}
    -\left\|w^\ell-\Bar{y}\right\| 
    & 
    \leq \liminf_k \left\|y^k -\Bar{y}\right\|-\lim _k \left\|y^k -w^\ell\right\| \\
    & \leq  \limsup_k
     \lV y^k - \Bar{y} \rV -\lim _k \| y^k -w^\ell \|  \leq\left\|w^\ell-\Bar{y}\right\|,
    \end{align}
    because, as aforementioned, item \textbf{(ii)} says that, for every  $\ell \in \na$,  
    $\lim_k\left\|w^\ell-y^k\right\|$ exists.
Now, as $\ell$ goes to infinity, we have $\left\|w^\ell-\Bar{y}\right\|  \to 0$, and we conclude that \[\liminf_k \left\|y^k -\Bar{y}\right\| = \limsup_k \left\|y^k -\Bar{y}\right\|=\lim_k\left\|y^k-\Bar y\right\|,\]
so that we have the desired result.
\end{proof}

Now, we show the extension of the Fej\'er monotonicity results needed for our analysis. We recall first the notions of R-linear and Q-linear convergence. 
A sequence $\{y^k\}\subset\re^n$ converges \emph{Q-linearly} to $y^\star\in\re^n$ if 
\[
\limsup_{k\to\infty}\frac{\lV y^{k+1}-y^\star\rV}{\lV y^k-y^\star\rV} =\eta< 1,
\]
and \emph{R-linearly} if
\[
\limsup_{k\to\infty}\lV y^k-y^\star\rV^{1/k}=\eta < 1.
\]
The value $\eta$ is said to be the \emph{asymptotic constant}. It is well known that Q-linear convergence implies R-linear convergence with the same asymptotic constant \cite{Ortega:2000}.

\begin{theorem}[Fej\'er* monotonicity and closed convex sets]
    \label{thm.Fejer*monotone_theorem}
    Suppose that $C \subset \re^n$ is a closed convex set with nonempty interior, that is, $\inte(C) \neq \emptyset$.  Assume that $\{y^k\}\subset\re^n$ is Fej\'er* monotone with respect to $\inte(C)$. Then,
    \begin{listi}
\item  if there exists a cluster point $\Bar{y}$ of $\{y^k\}$ which belongs to $C$, then the whole sequence $\{y^k\}$ converges to $\Bar y$;

\item  if $\{y^k\}$ converges to $y^\star\in C$ and $\{\dist(y^k,C)\}$ converges Q- or R-linearly to $0$ with asymptotic constant $\eta$, then  $\{y^k\}$ converges R-linearly to $y^\star$ with the same asymptotic constant.

\end{listi}
\end{theorem}
\begin{proof}

Item \textbf{(i)} is a direct consequence of \Cref{lem.Fejer*monotone.char}(iv), with $\inte(C)$ playing the role of $M$. 
Indeed,  let $\Bar{y}\in C$ be a cluster point of the bounded (due to \Cref{lem.Fejer*monotone.char}(i)) sequence $\{y^k\}$ and say $\{y^{\ell_k}\}$ is a subsequence of $\{ y^k\}$ converging to $\Bar{y}$, that is, $\norm{y^{\ell_k} - \Bar{y}}\to 0$. Note that $C = \closure{\conv(\inte(C))}$, and in view of \Cref{lem.Fejer*monotone.char}(iv), the sequence $\left\{\left\|y^k-\Bar{y}\right\|\right\}$ 
converges. Therefore,
\[
    \lim_{k\to \infty}   \norm{y^k-\Bar{y}} = \lim_{\ell_k\to \infty}   \norm{y^{\ell_k}-\Bar{y}} = 0,
\]
so, the whole sequence $\{y^k\}$ converges to $\Bar{y}$.

For proving \textbf{(ii)} note that Fej\'er* monotonicity of $\{y^k\}$ implies that for each $\hat y\in \inte(C)$ there exists $N(\hat y)$  such that $\lV y^{k+m}-\hat y\rV\le\lV y^k-\hat y\rV$ for all $k\ge N(\hat y)$ and all $m\in\na$. Hence,
\[
\label{eq.thm.Fejer*monotone.1}
    \|y^{k+m}-y^k \| \leq \|y^{k+m} - \hat y \| + \|y^k-\hat y\| \leq 2 \| y^k -\hat y\|,
\]
for all $\hat y\in \inte(C)$, all $k \geq N(\hat y)$ and all $m\in\na$.
Suppose $y_k \to y^{\star}$.  Taking limits in \cref{eq.thm.Fejer*monotone.1} with $m\to\infty$, we get
\[\label{e1}
\lV y^k-y^\star\rV\le 2\lV y^k -\hat y\rV, 
\] 
for all $\hat y\in \inte(C)$ and $k\ge N(\hat y)$.
Let $z^k=P_C(y^k)$,
so that $z^k\in C$. Convexity of $C$ ensures the existence of $\hat y\in \inte(C)$ such that $\lV \hat y-z^k\rV\le\lV y^k-z^k\rV$. Hence, in view of \cref{e1},
\begin{align}
\lV y^k-y^\star\rV & \le 2\lV y^k -\hat y\rV 
\le 2\lV y^k -\hat y\rV \\
& \le 2\lV y^k-z^k\rV+2\lV \hat y-z^k\rV \\ 
& \le 2\lV y^k-z^k\rV+2\lV y^k-z^k\rV \\ 
& = 4\lV y^k-z^k\rV =4\,\dist(y^k,C).
\end{align}
Taking $k$-th roots in both sides of the last inequality, we get 
\[\label{e3}
    \lV y^k-y^\star\rV^{1/k}\le 4^{1/k}\dist(y^k,C)^{1/k}.
\]
In view of the relation between Q-linear and R-linear convergence,
\[
\limsup_{k\to\infty}\lV y^k-y^\star\rV^{1/k}\le \limsup_{k\to\infty}\dist(y^k,C)^{1/k}=\eta<1,
\]
and we are done.
\end{proof}

\section{Statement of the algorithm}\label{s3}

We are now ready to introduce the Perturbed Approximate Circumcenter Algorithm (PACA) for the CFP defined in \cref{eq.def.CFP_functionValues}. Let us assume that the sets $C_i$, $i=1,\ldots,m$ in the CFP \cref{eq.CFP_1} take the form described in \cref{eq.def.CFP_functionValues}. Given an arbitrary initial point $x^0 \in \re^n$ and a decreasing sequence of perturbation parameters $\{\epsilon_k\}$ which converges to $0$, the PACA scheme for solving CFP is defined  in \Cref{algo:PACA}.

Next, we provide a somewhat informal explanation of several properties of the sequence $\{x^k\}$ generated by PACA. By examining \cref{eq.def.v_i_k} with $\epsilon_k=0$, and considering \Cref{lem.projection_to_halfspace}, it becomes evident that $x^k-v^{k}_{i}$ represents the orthogonal projection of $x^k$ onto a halfspace containing $C_i$, while $x^k-w^k$ is a convex combination of such projections. With the presence of perturbation parameters $\epsilon_k$, $x^k-v^{k}_{i}$ corresponds to the projection of $x^k$ onto the perturbed set $S^k_i$, defined as:
\[
\label{eq.def.s_i_k}
        S_i^k = 
        \begin{dcases*}
            C_i^k, &  if $x^k \in C_i^k$, \\
            \{z\in \re^n\mid (u_i)^\top(z-x^k)+f_i(x^k) + \epsilon_k \leq 0\}, & otherwise,
        \end{dcases*}
\]
which contains the perturbed convex set 
\[\label{a1}
C_i^k\coloneqq \{x\in\re^n\mid f_i(x)+\epsilon_k\le 0\},
\]
as we will prove in \Cref{lem.CK_SK}.

\begin{algorithm}
    \caption{Perturbed Approximate Circumcenter Algorithm (PACA)}
    \label{algo:PACA}
\begin{itemize}
\item[1)] {\bf Initialization.}

\noindent Take $x^0\in\re^n$.
\item[2)] {\bf Iterative step.}

Given $x^k$, take $u_{i}^{k}\in\partial f_i(x^k)$, the subdifferential of $f_i$ at $x^k$ 
(so, $u_{i}^{k}$ is a subgradient of $f_i$ at $x^k$), and define:\[
    \label{eq.def.v_i_k}
    v^{k}_{i} \coloneqq \left[\frac{\max\{0,f_i(x^k)+\epsilon_k\}}{\lV u_{i}^{k}\rV^2}\right]u_{i}^{k},
\]
\[
    \label{eq.def.w_k}
    w^k \coloneqq \frac{1}{m} \sum_{i=1}^m v^{k}_{i}.
\]
If $w^k=0$, then take $x^{k+1}=x^k$ and proceed to the $(k+1)$-th iteration. Otherwise, define
\[
    \label{eq.def.alpha_k}
\alpha_k \coloneqq \frac{\frac{1}{m}\sum_{i=1}^m\lV v^{k}_{i}\rV^2}{\lV w^k\rV^2},
\] 
\[
    \label{eq.def.PAC}
    x^{k+1} \coloneqq x^k-\alpha_k w^k.
\]

\item[3)] {\bf Stopping criterion.}

If $x^k \in C \coloneqq \bigcap_{i=1}^m C_i$, then stop (in this case, we say that PACA has finite convergence).
\end{itemize}
\end{algorithm}

We mention that,
depending on the value of $\epsilon_k$, the set $C^k_i$ may be empty. On the other hand, $S^k_i$ is always nonempty, so that the sequence $\{x^k\}$ is always well-defined, independently of the value of $\epsilon_k$. 
However, the set $S^k\coloneqq \bigcap_{i=1}^mS^k_i$ may still be empty, in which case the sequence $\{x^k\}$ may exhibit an erratic behavior. It could even happen that $w^k=0$, in which case, as stated in the Iterative Step, we get $x^{k+1}=x^k$.
This only occurs if, by chance, $x^k$ minimizes the function $\sum_{i=1}^m \| x-P_{S^k_i}(x)\|^2$. 
In this case the sequence does not stop: it proceeds to the next iteration, with perturbation parameter $\epsilon_{k+1}<\epsilon_k$. For large enough $k$, however,  the set $S^k$ will become nonempty; this is ensured to happen when the Slater point $\hat x$ gets inside $S^k$ (more precisely, when $k$ is large enough so that $\epsilon_k<\min_{1\le i\le m} \{\lv f_i(\hat x)\rv\}$). At that point, the sequence starts moving toward $S^k$, and \emph{a fortiori} toward $C^k\coloneqq \bigcap_{i=1}^mC^k_i\subset S^k$, so that eventually it converges to a point in $C$. We also observe that $x^k-w^k$ is a convex combination, with equal weights, of the projections onto the perturbed sets $S^k_i$. 

Note that, when $\epsilon_k =0$, the set $S^k$ is related to the surrogate halfspaces studied and used in the algorithms proposed in \cite{Kiwiel:1995, Kiwiel:1997,Combettes:2000}. In particular, the original algorithm proposed in \cite{Kiwiel:1995} used a relaxed projection onto the surrogate (and unperturbed) block $S_k$, with only a part of the $S_i^k$  used at each iteration. Moreover, in the case of PACA, the surrogate sets $S_i^k$ are perturbed, and the next iteration is computed using the circumcenter scheme. Indeed, the factor $\alpha_k$ in the definition of $x^{k+1}$ (\emph{i.e.}, in \cref{eq.def.PAC}) encapsulates the contribution of the circumcentered-reflection approach, as we will discuss in the next section. 
It is easy to verify that if $\epsilon_k = 0$ for all $k$, then the PACA sequence coincides with the CARM sequence in \cite{Araujo:2022}, using the separating operator given in \cite[Ex.~2.7]{Araujo:2022}. If we remove the factor $\alpha_k$ in \cref{eq.def.PAC}, PACA reduces to Algorithm 4 in \cite{Iusem:1986a} with relaxation parameters equal to $1$. In this case, $x^{k+1}$ is simply a convex combination of perturbed approximate projections of $x^k$ onto the sets $C_i$.
We mention that the algorithm in \cite{Iusem:1986a} also enjoys finite convergence, under the same assumptions as in \Cref{thm.finite_convergence} in this paper. It is worth mention that, analogous algorithms to ours and the one in \cite{Iusem:1986a}  were considered in \cite[Ex.~3.2]{Kolobov:2021} and in  \cite[Ex.~4.9]{Kolobov:2022}.


\begin{proposition}
\label{cor.boundedness_alphak}
Consider the PACA sequence defined by \cref{algo:PACA}. Then,  
$\alpha_k\ge 1$ for all $k$.
\end{proposition}

\begin{proof}
The result follows directly from
\cref{eq.def.w_k,eq.def.alpha_k}, and the convexity of the function $\lV\cdot \rV^2$.
\end{proof}

\Cref{cor.boundedness_alphak} indicates that the PACA sequence can be seen as an overrelaxed version of the sequence given by $x^{k+1}=x^k-w^k$, studied in \cite{Iusem:1986a}. We comment that the particular value of $\alpha_k$ determined in \cref{eq.def.alpha_k}, related to the circumcenter approach,  has special consequences in terms of the performance of the method.
Moreover,  the relaxation $\alpha_k$ can be equivalently interpreted as the extrapolation parameter for the simultaneous projection $\frac{1}{m} \sum_{i=1}^m P_{S_i^k}$, where the $S_i^k$'s are perturbed versions of the original sublevel sets; see \cref{eq.def.s_i_k}. Indeed, we have $v^{k}_i=$ $x^k-P_{S_i^k}\left(x^k\right)$. Consequently, \[w^k=\frac{1}{m} \sum_{i=1}^m v^{k}_i=\frac{1}{m} \sum_{i=1}^m\left(x^k-P_{S_i^k}\left(x^k\right)\right),\]and
\[
\alpha_k=\frac{\frac{1}{m} \sum_{i=1}^m\left\|v^{k}_i\right\|^2}{\left\|w^k\right\|^2}=\frac{\frac{1}{m} \sum_{i=1}^m\left\|x^k-P_{S_i^k}\left(x^k\right)\right\|^2}{\left\|\frac{1}{m} \sum_{i=1}^m\left(x^k-P_{S_i^k}\left(x^k\right)\right)\right\|^2} .
\]
An analogous expression appears in the extrapolated method of parallel projections; 
see \cite[eq.~(1.9)]{Combettes:1997b}. A comprehensive study of extrapolation techniques 
for iterative methods  can be found in \cite[sects.~4.9 and 5.10]{Cegielski:2012}.

\section{Convergence Analysis}\label{s4} For the convergence analysis, it is necessary to establish a connection between PACA and the CRM method applied to a CFP in the product space $\re^{nm}$. This allows us to utilize Pierra's formulation, which transforms a CFP with $m$ sets in $\re^n$ into a CFP with only two sets in $\re^{nm}$, with one of them being a linear subspace.

Consider $C^k_i$ as defined in \cref{a1} (once again, we recall that $C^k_i$ may be 
empty for some values of $\epsilon_k$) and $S_i^k$, as defined in \cref{eq.def.s_i_k}.
Using \Cref{lem.projection_to_halfspace}, the projection of $x^k$ onto  $S_i^k$ is given by:
\[
\label{eq.def.Approximate_projection_onto_C_epsilon}
 P_{S_i^k}(x^k) = x^k - \frac{\max\{0,f_i(x^k)+\epsilon_k\} } {\|u_{i}^{k}\|^2} u_{i}^{k}, 
\]
with $u_{i}^{k} \in \partial f_i(x^k)$. \emph{i.e.}, $u_{i}^{k}$ is a subgradient of $f_i$ at $x^k$. We remark that
\cref{eq.def.Approximate_projection_onto_C_epsilon} also works when $x^k$ belongs to $C^k_i$, in which case $S_i^k$ is not a halfspace, but just the set $C_i^k$, and  
$P_{S_i^k}(x^k) = x^k$.

Now we introduce the appropriate sets in $\re^{nm}$ for using
Pierra's formulation. Define: 
\[
    \label{eq.def.bf_S_k}
    \bSS^k = S_1^k\times \dots \times S_m^k,
\]
and 
\[
    \label{eq.def.bf_C_k}
    \CC ^k = C_1^k \times \dots \times C_m^k.
\]
We need an elementary result on the relation between $\CC ^k$ and $\bSS^k$.
\begin{proposition}
\label{lem.CK_SK}
    Take $\bSS^k,\CC ^k \subset \re^{nm}$ as defined in \cref{eq.def.bf_C_k,eq.def.bf_S_k}, respectively. Then,
    \[
    \label{eq.relationship_Ck_Sk}
        \CC ^k \subset \bSS^k.
    \]
\end{proposition}
\begin{proof}
The result holds trivially if $\CC ^k$ is empty (\emph{i.e.}, if $C_i^k=\emptyset$ for some $i$). Otherwise, it suffices  that $C_i^k\subset S_i^k$ for all $i$. Take any point $z \in C_i^k$, so that $f_i(z) +\epsilon_k \leq 0$. If $x^k \in C_i^k$, then the result is satisfied by the definition of $S_i^k$. If $x^k \notin S_i^k$, since the $f_i$'s are convex, then we have
    \[\label{a2}
        (u_{i}^{k})^\top(z-x^k) + f_i(x^k ) +\epsilon_k \leq f_i(z)+ \epsilon^k \leq 0, 
    \]
    for all $z\in\re^n$, using \cref{a1} and the definition of subgradient. The result follows from \cref{a2} and the definition of $S^k_i$ (\emph{i.e.}, \cref{eq.def.s_i_k}), since $z\in S_i^k$ for all $i$.
\end{proof}

Recall that  $\DD\subset\re^{nm}$ is the \emph{diagonal subspace} in $\re^{nm}$.
In the next lemma, we will prove that the sequence $\{\vx^k\}\subset\re^{nm}$ generated by {CRM} with the sets $\bSS^k$ and $\DD$ is related to the sequence $\{x^k\}\subset\re^n$ generated by PACA with the sets $C^k_i$. In fact, we will get that $\vx^k=(x^k,x^k,\dots ,x^k)$ for all $k$.  

\begin{lemma}
\label{lem.PAC_and_CRM}
    Assume that, for $i=1,\ldots, m$, $C_i\subset \re^n$ is in the form of \cref{eq.def.CFP_functionValues} with each $f_i:\re^n\to\re$ convex. Consider the set $\bSS^k$, as defined in \cref{eq.def.bf_S_k}, the diagonal set $\DD$, and the sequence $\{\vx^k\}\subset\re^{nm}$  generated by CRM given in \cref{eq.definition_Circumcenter} and starting from a point $\vx^0  \in \DD$, 
    \emph{i.e.},
        \[
        \label{eq.def.CRM_iteration}
        \vx^{k+1} \coloneqq \CRMOp_{\bSS^k ,\DD}(\vx^k) = \circum(\vx^k,R_{\bSS^k}(\vx^k),R_{\DD}(R_{\bSS^k}(\vx^k))),
    \]
    where
    $R_{\bSS^k} \coloneqq 2P_{\bSS^k} - \Id$, $R_\DD \coloneqq 2P_\DD-\Id$. Let $\{x^k\} \subset \re^n$ be generated by PACA as in \cref{eq.def.PAC} starting from $x^0 \in \re^n$. 
    If $\vx^0 = (x^0,\ldots,x^0)$, then $\vx^k = (x^k,\ldots,x^k)$ for all $k\geq 1$. Also, 
    PACA stops at iteration $k$ if and only if CRM does so.   
\end{lemma}
\begin{proof}
    Since $\vx^0 \in \DD$, which is an affine subspace, and  $\bSS^k$ is a closed convex set, we get (inductively) that, for all $k\geq 0$, not only $\vx^k$ is computable (therefore, sequence $\{\vx^k\}$ is well-defined) but also  $\vx^k \in \DD$; see \Cref{lem.CRM}(ii). 

   Suppose now that $\vx^k \coloneqq (y^k,\ldots,y^k)$ with $y^k \in \re^n$. Since $\vx^0\coloneqq (x^0, \dots ,x^0)$, we assume inductively that $y^k=x^k$, and we must prove that $\vx^{k+1}=(x^{k+1}, \dots ,x^{k+1})$. We proceed to compute the arguments of the right-hand side of \cref{eq.def.CRM_iteration} in order to get $\CRMOp_{\bSS^k ,\DD}(\vx^k)$.
    By \cref{eq.def.bf_S_k},  we get
    \[\label{e5}
    P_{\bSS^k}(\vx^k)=\left(P_{S^k_1}(x^k),\dots ,P_{S^k_m}(x^k)\right).
    \]
    Moreover, \cref{eq.def.Approximate_projection_onto_C_epsilon} and
    \cref{eq.def.v_i_k} yields
    \[\label{e6}
    P_{S_i^k}(x^k) = x^k -\frac{\max\{0,f_i(x^k)+\epsilon_k\}} {\lV u_{i}^{k}\rV^2}u_{i}^{k}=x^k-v^{k}_{i},
    \]
    with $u_{i}^{k}\in \partial f_i(x^k)$.
    It follows from \cref{e5} and \cref{e6} that
    \[\label{e7}
     P_{\bSS^k}(\vx^k)=\left(x^k-v^{k}_1,\dots ,x^k-v^{k}_m\right).
     \]
     In view of \cref{e7} and of the definition of the reflection operator, we have 
     \[\label{e8}
     R_{\bSS^k}(\vx^k)=\left(x^k-2v^{k}_1,\dots ,x^k-2v^{k}_m\right).
     \]
     It follows from the definition of the diagonal subspace $\DD$ that $P_\DD(x^1, \dots ,x^m)=(z, \dots, z)$
     with $z=(1/m)\sum_{i=1}^mx^i$. 
     Hence, in view of \cref{eq.def.w_k},
     \[\label{e9}
     P_\DD(\vx^k)=(x^k-2w^k, \dots ,x^k-2w^k).
     \]
     Finally, in view of \cref{e9},
     \[\label{e10}
     R_\DD\left(R_{\bSS^k}(\vx^k)\right)=(x^k-4w^k, \dots ,x^k-4w^k).
     \]
    
     Now, by the definition of circumcenters, 
    $\CRMOp_{\bSS^k ,\DD}(\vx^k)$
    must be equidistant to $\vx,\vy$ and $\vz$,
    with $\vy\coloneqq R_{\bSS^k}(\vx)$ as in \cref{e8}  and
    $\vz \coloneqq R_\DD\left(R_{\bf S^k}(\vx^k)\right)$
    as in \cref{e10}. 
    An elementary but somewhat lengthy calculation using \cref{e8} and \cref{e10} shows that any vector of the form 
    ${\bf e}(\alpha) =(x^k-\alpha w^k, \dots ,x^k-\alpha w^k)$ with $\alpha\in\re$ is equidistant from $\vy$ and $\vz$. 
    
    Note that the equation    $\lV{\bf e}(\alpha) -\vx\rV^2=\lV {\bf e}(\alpha)-\vz\rV^2$, leads, after some elementary calculations, to 
    a linear equation in $\alpha$, whose solution is
    $\alpha=\alpha_k$, with $\alpha_k$ as in \cref{eq.def.alpha_k}. Hence, taking into account 
    \cref{eq.def.PAC},
    we get
    \[
    \vx^{k+1}=\CRMOp_{\bSS^k ,\DD}(\vx^k)
    =(x^k-\alpha_k w^k, \dots ,x^k-\alpha_k w^k)=(x^{k+1}, \dots ,x^{k+1}),
    \]
    completing the inductive step. The equivalence between the stopping criteria of both algorithms is immediate.
    \end{proof}

    Another useful result is the quasi-nonexpansiveness of the CRM operator $\CRMOp_{\bSS^k ,\DD}$, with respect to $\CC^k$.

 \begin{lemma}[Quasi-nonexpansiveness of $\CRMOp_{\bSS^k ,\DD}$]
    \label{cor.nonexpansiveness_CRM}
    Let $\bSS^k,\CC ^k \subset \re^{nm}$ be defined as in \cref{eq.def.bf_C_k,eq.def.bf_S_k}, respectively. 
    Consider the operator $\CRMOp_{\bSS^k ,\DD}$ as defined in \cref{eq:def.CRM_operator} and assume that $\CC ^k\ne\emptyset$. Then,
    \[
        \label{eq.cor.nonexpansiveness_CRM}
        \lV \CRMOp_{\bSS^k ,\DD}(\vx)-{\bf s}\rV^2\le\lV\vx-{\bf s}\rV^2-\lV \CRMOp_{\bSS^k ,\DD}(\vx)-\vx\rV^2,
    \]
    for all $\vx\in\re^{nm}$ and all ${\bf s} \in \CC ^k$.
\end{lemma}

\begin{proof}
Since $S_i^k$ is closed and convex by definition, then $\bSS^k$ is also closed and convex. The result follows immediately from \Cref{lem.CRM}(iii) and \Cref{lem.CK_SK}, with $\bSS^k$ playing the role of $A$ and $\DD$, the role of $B$.     
\end{proof}

Next, we prove that the sequence generated by PACA is Fej\'er* monotone with respect to $\inte(C)$.

\begin{lemma}
\label{lem.Fejer*monotoness_PAC}
Consider the  CFP with the $C_i$'s of the form \cref{eq.def.CFP_functionValues}. Suppose that $\{ x^k\}\subset\re^n$ is the sequence generated by PACA starting at $x^0 \in \re^n$. If the sequence $\{x^k\}$ is infinite, then,
\begin{listi}
\item 
$\{x^k\}$ is Fej\'er* monotone with respect to $\inte(C)$;
\item $\lim_{k\to\infty}\lV x^{k+1}-x^k\rV=0$;
\item
The sequence $\{x^k\}$ is bounded.
\end{listi}
\end{lemma}
\begin{proof}
For item {\bf(i)}, let $\{\vx^k\}\subset\re^{nm}$ be the sequence generated by CRM as in \cref{eq.def.CRM_iteration} starting at $\vx^0 = (x^0,\ldots,x^0)$.
By definition of $\vx^0$, it belongs to $\DD \subset \re^{nm}$. By \Cref{lem.CRM}(ii) and \cref{eq.def.CRM_iteration}, $\vx^k \in \DD$ for all $k$. Define 
\[\label{e13}
C^k\coloneqq \bigcap_{i=1}^m C_i^k.
\]
From \Cref{cor.nonexpansiveness_CRM} and \cref{eq.def.CRM_iteration}, we get
\[
    \label{eq.lem.Fejer*monotoness_PAC.5}
\lV\vx^{k+1} -{\bf s}\rV^2\le\lV\vx^k-{\bf s}\rV^2-\lV\vx^{k+1}-\vx^k\rV^2,
\]
for all ${\bf s} \in \CC ^k\cap \DD$. Since $\{\vx^k\}\subset\DD$ and ${\bf s}\in\DD$, we
have, in view of \Cref{lem.PAC_and_CRM}, $\vx^k=(x^k, \dots ,x^k)$, $\vx^{k+1}=(x^{k+1}, \dots,x^{k+1})$ and ${\bf s} =(s,\dots ,s)$ with $s\in C^k\subset\re^n$ as defined in \cref{e13}, so that
$\lV \vx^{k+1}-{\bf s}\rV^2= m\lV x^{k+1}-s\rV^2$, $\lV \vx^k-{\bf s}\rV^2=m\lV x^k-s\rV^2$ and $\lV \vx^{k+1}-\vx^k\rV^2=m\lV x^{k+1}-x^k\rV^2$. Hence, it follows from \cref{eq.lem.Fejer*monotoness_PAC.5} that
\[
         \label{eq.lem.Fejer*monotoness_PAC.6}
          \lV x^{k+1}-s\rV^2\le\lV x^k-s\rV^2-\lV x^{k+1} -x^k\rV^2,
 \]
for all $s \in C^k$ as defined in \cref{e13}. The Slater condition implies that $\inte(C)\ne\emptyset$. Take any $s\in\inte(C)$,
so that $f_i(s)< 0$ for all $i$. Since $\{\epsilon_k\}$ converges to $0$, there exists $N(s)\in\na$, such that $\epsilon_k < \min_{1\leq i \leq m} \lv f_i(s)\rv$ for all $k\geq N(s)$. Therefore, we have $s\in \CC ^k$ for all $k\geq N(s)$. Hence, in view of \cref{eq.lem.Fejer*monotoness_PAC.6}
\[
         \label{e11}
          \lV x^{k+1}-s\rV^2\le\lV x^k - s\rV^2,
     \]
     for all $s \in \inte(C)$ and all $k \geq N(s)$. In view of the definition of Fej\'er* monotonicity, $\{x^k\}$ is Fej\'er* monotone with respect to $\inte(C)$.

      For \textbf{(ii)}, note that \cref{eq.lem.Fejer*monotoness_PAC.6} implies that $\left\{\lV x^k-s\rV\right\}$ is decreasing and non-negative. Hence, $\left\{\lV x^k-s\rV\right\}$ is convergent, and also
     \[\label{e12}
     \lV x^{k+1}-x^k\rV^2\le\lV x^k-s\rV^2-\lV x^{k+1}-s\rV^2,
     \]
     which immediately implies the result.

     For item {\bf (iii)}, notice that the Slater condition implies that $\inte(C)\ne\emptyset$. Then the result follows from item (i) and \Cref{lem.Fejer*monotone.char}(i).
\end{proof}

Next we prove that the cluster points of the sequence generated by PACA solve the CFP \cref{eq.CFP_1}.

\begin{proposition}
\label{lem.clusterpoit_in_C}
Let $\{x^k\}$ be the sequence generated by PACA. If  $\{x^k\}$ is infinite and $\Bar{x}$ is a cluster point of $\{x^k\}$, then $\Bar{x} \in C$.
\end{proposition}
\begin{proof}
First, note that the existence of cluster points of $\{x^k\}$ follows from \Cref{lem.Fejer*monotoness_PAC}.
Let $\hat x$ be a Slater point for the CFP. Since all the $f_i$'s are convex and 
$f_i(\hat x)< 0$ for all $i$, we get from the definition of subgradient that 
\[
    \label{eq.lem.clusterpoit_in_C.1}
    (u_{i}^{k})^\top (\hat{x}-x^k) \leq f_i(\hat{x}) - f_i(x^k) \leq - f_i(x^k),
\]
Let $\{x^{\ell_k}\}$ be a subsequence of $\{x^k\}$ which converges to the cluster point $\Bar x$. By \cref{eq.def.PAC},
\[\label{e15}
x^k-x^{k+1}=\alpha_kw^k.
\]
By \Cref{cor.boundedness_alphak}, we know that $\alpha_k \geq 1$, so
\[\label{eq.lem.clusterpoit_in_C.3}
\left\|x^k-x^{k+1}\right\|^2=\alpha_k^2\left\|w^k\right\|^2=\alpha_k\left(\frac{1}{m} \sum_{i=1}^m\left\|v^{k}_i\right\|^2\right) \geq \frac{1}{m} \sum_{i=1}^m\left(\frac{\max \left\{0, f_i\left(x^k\right)+\epsilon_k\right\}}{\left\|u_i^{ k}\right\|^2}\right)^2. 
\]

Since $\{x^k\}$ is bounded by \Cref{thm.Fejer*monotone_theorem}(i) and the subdifferential is locally bounded (see, \emph{e.g.}, \cite[Thm.~24.7]{Rockafellar:1997}), we refine the subsequence $\{x^{\ell_k}\},$ if needed, in order to  ensure that for each $i$, $\{u_i^{\ell_k}\}$ converges, say to $\Bar{u}_i$. Let $I\coloneqq\{i \in \{1,
\ldots, m\}\mid  f_i(\bar x)\ge 0\}$. We take limits with $k\to\infty$ on both sides of \cref{eq.lem.clusterpoit_in_C.3} along the  refined subsequence $\{x^{\ell_k}\}$, getting
\begin{equation}\label{e56}
    0\ge\frac{1}{m}\sum_{i\in I}\left(\frac{f_i(\bar x)}{\lV\bar u_i\rV^2}\right)^2,
\end{equation}
because $\lim_{k\to\infty}\epsilon_k=0$, and $\lim_{k\to\infty} (x^{k+1}-x^k)=0$ by \Cref{lem.Fejer*monotoness_PAC}(ii).
It follows from \cref{e56} that $f_i(\Bar x)=0$ for all $i\in I$. Since, by the definition of $I$, $f_i(\Bar x) < 0$ for all $i\notin I$, we get that $f_i(\Bar x)\le 0$ 
for $i=1,\ldots,m$. Therefore, we conclude that $\bar x\in C$.
\end{proof}

We close this section by establishing convergence of the full sequence generated by PACA to a solution of the CFP  \cref{eq.CFP_1}.

\begin{theorem}
\label{thm.infiniteConvergence}
 Let $\{x^k\}$ be the sequence generated by PACA starting from an arbitrary point $x^0 \in \re^n$. If $\{x^k\}$ is infinite, then it converges to some $x^\star\in C$.
\end{theorem}
\begin{proof}
By \Cref{lem.Fejer*monotoness_PAC}(iii), $\{x^k\}$ is bounded, so that it has cluster points. \Cref{lem.clusterpoit_in_C} yields that its cluster points belong to $C$. Finally, \Cref{thm.Fejer*monotone_theorem}(i) implies that  $\{x^k\}$ converges to some $x^\star\in C$.
\end{proof}

\section{Linear convergence rate}\label{s5}

In this section we will prove that if the sequence generated by PACA is infinite, then under the Slater condition it enjoys a linear convergence rate. It is worth mentioning that in other papers addressing circumcentered-reflection methods, such as  \cite{Behling:2024,Arefidamghani:2021,Behling:2018a,Araujo:2022}, a linear rate of the generated sequence was proved under an error bound (or transversality) assumption~\cite{Kruger:2018a,Behling:2021}. This assumption is somewhat weaker than our Slater condition (see Proposition 3.2, Theorem 3.3 and the last paragraph of Subsection 3.1 in \cite{Behling:2024}) but under the Slater condition we can achieve a better result, namely finite convergence. In the above-mentioned references the asymptotic constant related to the linear convergence depends on some constants linked to the
error bound assumption; as it could be expected, our asymptotic constant is given in terms of the Slater point. 

\begin{proposition}
\label{cor.cor1}
Let $\{x^k\}$ be the sequence generated by PACA. If $\{x^k\}$ is infinite, then there exist $N\in\na$ such that
\[
\label{eq.cor1}
\dist^2(x^{k+1},C^k) \le \dist^2(x^k,C^k)-\lV x^{k+1}-x^k\rV^2,
\]
for all $k\geq N$.
\end{proposition}
\begin{proof}
We have proved in \cref{eq.lem.Fejer*monotoness_PAC.6} that
\[\label{e16}
\lV x^{k+1}-s\rV^2\le\lV x^k-s\rV^2-\lV x^{k+1} -x^k\rV^2,
 \]
for all $k\geq N$ and all $s\in C^k$. The result follows by taking minimum with $s\in C^k$ in both sides of \cref{e16}.  
\end{proof}

\begin{proposition}
\label{lem.dist_functionValue}
Let $\{x^k\}$ be the sequence generated by PACA. If $\{x^k\}$ is infinite, then there exist $N\in\na$ and $\beta\in\re$ 
such that
    \[
    \label{eq.lem.dist_functionValue}
        \dist(x^k,C^k) \le\beta(f(x^k)+\epsilon_k),
    \]
for all $k\ge N$, where  $f(x)\coloneqq \max_{1\le i \le m} f_i(x)$.
\end{proposition}
\begin{proof}

It follows from the definition of $f$  that $C^k = \left\{x\in \re^n \mid  f(x)+\epsilon_k \leq 0\right\}$. Since $\left\{\epsilon_k\right\}$ is decreasing, there exists $N\in \na$ such that $C^k \neq \emptyset$, for all $k \geq N$. Moreover, knowing that $\left\{x^k\right\}$ is infinite, we must have $f\left(x^k\right)>0$. Consequently, thanks to \cite[Lem.~2.1]{Kolobov:2021}, we get
\[
\frac{\dist\left(x^k, C^k\right)}{f\left(x^k\right)+\epsilon_k} \leq \frac{\dist\left(x^k, C^N\right)}{f\left(x^k\right)+\epsilon_N}.
\]
The inclusion $C^N \subset C^k$, when combined with \cref{eq.lem.Fejer*monotoness_PAC.6}, leads to $\left\|x^k-s\right\| \leq\left\|x^N-s\right\|$ for all $k \geq N$ and for all $s \in C^N$. In particular, we get $\dist\left(x^k, C^N\right) \leq \dist\left(x^N, C^N\right)$. 

On the other hand, $f\left(x^k\right)+\epsilon_N > \epsilon_N$. Consequently, we arrive at
\[
\dist\left(x^k, C^k\right) \leq \underbrace{\frac{\dist\left(x^N, C^N\right)}{\epsilon_N}}_\beta\left(f\left(x^k\right)+\epsilon_k\right) .
\]
\end{proof}

\begin{proposition}
\label{cor.cor2}
Let $\{x^k\}$ be the sequence generated by PACA, starting from $x^0\in\re^n$. If  $\{x^k\}$ is infinite, then there exists $N\in\na$ and  $\sigma>0$ such that
\[
    \label{eq.cor.cor2}
        \dist^2(x^k,C^k)\le m(\beta\sigma)^2\lV x^{k+1}-x^k\rV^2,
    \]  
for all $k\ge N$.
\end{proposition}
\begin{proof}
Since $\{x^k\}$ is bounded, and the subdifferential of a convex function defined on the whole $\re^n$ is locally bounded (again, by \cite[Thm.~24.7]{Rockafellar:1997}), we can take $\sigma\in\re$ such that
\[\label{e26}
\lV u_{i}^{k}\rV\le\sigma,
\]
for all $k$ and all $i$. If needed, we also take $\sigma$ large enough so that $m(\beta\sigma)^2> 1$.

Let $\{\vx^k\}\subset\re^{nm}$ be the sequence generated by CRM as in \cref{eq.def.CRM_iteration} starting at $\vx^0 = (x^0,\ldots,x^0)$. Then,
\begin{align}
2\lV\vx^k - P_{\bSS^k}(\vx^k) \rV & =
\lV\vx^k-R_{\bSS^k}(\vx^k)\rV  \le\lV\vx^k-\vx^{k+1}\rV+\lV \vx^{k+1}-R_{\bf S^k}(\vx^k)\rV \\
& =2\lV\vx^{k+1}-\vx^k\rV,\label{e24} 
\end{align}
using the definition of reflection in the first equality, and the definition of circumcenter, together with \cref{eq.definition_Circumcenter} and \cref{eq:def.CRM_operator}, in the second one. It follows from \cref{e24} that
\[\label{e50}
\lV\vx^{k+1}-\vx^k\rV^2\ge
\lV\vx^k - P_{\bSS^k}(\vx^k)\rV^2.
\]

Now, define
\[
    \label{eq.def.I_k}
    I_k\coloneqq\{i \in\{1,\ldots,  m\} \mid f_i(x^k)+\epsilon_k >0 \} .
\]
In view of \Cref{lem.PAC_and_CRM}, we pass 
from $\re^{nm}$ 
to $\re^n$ in \cref{e50}, obtaining
\begin{align}
 m\lV x^{k+1}-x^k\rV^2 & \ge\sum_{i=1}^m\lV x^k - P_{S_i^k}(x^k)\rV^2=\sum_{i=1}^m\lV\frac{\max\{0,f_i(x^k)+\epsilon_k\}}{\|u_{i}^{k}\|^2}u_{i}^{k}\rV^2 \\ 
 &=\sum_{i=1}^m\left[\frac{\max\{0,f_i(x^k)+\epsilon_k\}}{\lV u_{i}^{k}\rV}\right]^2=\sum_{i\in I_k}\left[\frac{f_i(x^k)+\epsilon_k}{\lV u_{i}^{k}\rV}\right]^2\ge\left[\frac{f(x^k)+\epsilon_k}{\sigma}\right]^2,\label{e27}
 \end{align}
 where $ f(x)\coloneqq \max_{1\le i \le m} f_i(x)$.
The first equality in \cref{e27} holds by \cref{eq.def.Approximate_projection_onto_C_epsilon},
the third equality by \cref{eq.def.I_k}, and the last inequality follows from the definition of $f$ and \cref{e26}. 


Take $N$ as in \Cref{lem.dist_functionValue}.
Combining \cref{e27} with \Cref{lem.dist_functionValue}, we obtain
\[
\dist^2(x^k,C^k)\le \beta^2(f(x^k)+\epsilon_k)^2\le m(\beta\sigma)^2\lV x^{k+1}-x^k\rV^2,
\]
for all $k\geq N$, completing the proof.
\end{proof}

\begin{proposition}
\label{lem.linearConvergence_distanceSequence.fake}
Let $\{x^k\}$ be the sequence generated by PACA. If $\{x^k\}$ is infinite, then the sequence $\{\dist(x^k,C^k)\}$ converges Q-linearly to $0$.
\end{proposition}
\begin{proof}
    By \Cref{cor.cor1}, there exists $N\in\na$ such that
    \[\label{e28}
        \dist^2(x^{k+1},C^k)\le\dist^2(x^k,C^k) - \lV x^{k+1}-x^k\rV^2,
    \]
    for all $k\ge N$. Since $\{\epsilon_k\}$ is monotonically decreasing, we have $C^k\subset C^{k+1}$, so that we get from \cref{e28}, 
    \[\label{e29}
    \dist^2(x^{k+1},C^{k+1})\le\dist^2(x^k,C^k)-\lV x^{k+1}-x^k\rV^2.
    \]
    Using \Cref{cor.cor2} and \cref{e29}, we get
    \[\label{e33}
        \dist(x^{k+1},C^{k+1})\le\sqrt{1-\frac{1}{m(\beta\sigma)^2}}\,\dist(x^k,C^k),
    \]
    completing the proof.
\end{proof}
Next we prove that the sequence $\{\dist(x^k,C)\}$ converges R-linearly to $0$.
\begin{proposition}
\label{cor.linearConvergence_distanceSequence}
Let $\{x^k\}$ be the sequence generated by PACA. If  $\{x^k\}$ is infinite, then the sequence $\{dist(x^k,C)\}$ converges R-linearly to $0$.
\end{proposition}
\begin{proof} Let 
\[\label{e43}
\lambda\coloneqq\sqrt{1-\frac{1}{m(\beta\sigma)^2}},
\]
with $\beta,\sigma$ as in
\Cref{lem.linearConvergence_distanceSequence.fake}. By \Cref{lem.linearConvergence_distanceSequence.fake}, there exists $N\in\na$ such that, for $k\geq N$,
    \[
        \label{e30}
        \dist(x^k,C^k)\le\lambda^{k+1-N} \dist(x^{N-1},C^{N-1}).
    \]
    Since $C^k \subset C$ for all $k$, we get from \cref{e30},
    \[
        \label{e31}
        \dist(x^k,C)\le \dist(x^k,C^k)\le \lambda^{k+1-N} \dist(x^{N-1},C^{N-1}),
    \]
    for all $k\ge N$. Taking $k$-th roots in \cref{e31}
    and then limits with $k\to\infty$, we get
    \[\label{e32}
    \lim_{k\to\infty}[\dist(x^k,C)]^{1/k}=
    \lim_{k\to\infty}\lambda^{1-(N+1)/k}[\dist(x^N,C^N)]^{1/k}=\lambda.
    \]
  Since $\lambda<1$ by \cref{e43}, we conclude that $\{\dist(x^k,C)\}$
  converges R-linearly to $0$.
 \end{proof}

\begin{theorem}
\label{thm.linearconvergence_sequence}
Let $\{x^k\}$ be the sequence generated by PACA starting from an arbitrary point $x^0 \in \re^n$. If $\{x^k\}$ is infinite, then $\{x^k\}$ converges R-linearly to some point $x^\star\in C$.
\end{theorem}
\begin{proof}
The result is a direct consequence of \Cref{cor.linearConvergence_distanceSequence} and \Cref{thm.Fejer*monotone_theorem}(ii).
\end{proof}

\section{Finite Convergence}\label{s6}

In this section we will prove that under an additional assumption on the sequence of perturbation parameters $\{\epsilon_k\}$, PACA enjoys finite convergence, \emph{i.e.}, $x^k$ solves the CFP for some value of $k$. 

Our next result is somewhat remarkable, because it states that if the sequence $\{x^k\}$ is infinite then the sequence $\{\epsilon_k\}$ of perturbation parameters must be summable.
Now, the sequence $\{\epsilon_k\}$ is an exogenous one, which can be freely chosen as long as it decreases to $0$. If we select it so that it is not summable, then it turns out that the sequence $\{x^k\}$ cannot be infinite, and hence, in view of the stopping criterion, there exists some $k\in\na$ such that $x^k$ solves CFP.

We give now an informal argument which explains this phenomenon.
We have proved that $x^k$ approaches the perturbed set $C^k$
at a certain speed (say, linearly), independently of how fast
$\epsilon_k$ decreases to zero. On the other hand, the perturbed sets $C^k$ keep increasing, approaching the target set $C$ from the inside, with a speed determined by the $\epsilon_k$'s. If  $\epsilon_k$ goes to $0$ slowly (say, sublinearly), then, at a certain point, $x^k$ will get very close to $C^k$,
while $C^k$ is still well inside $C$. At this point, $x^k$ gets trapped in $C$, so that it solves CFP, the algorithm suddenly stops, and we get finite convergence. \cref{pic.PACA} depicts such behavior.

\begin{figure}[htpb]
    \centering
    {\includegraphics[width=.7\textwidth]{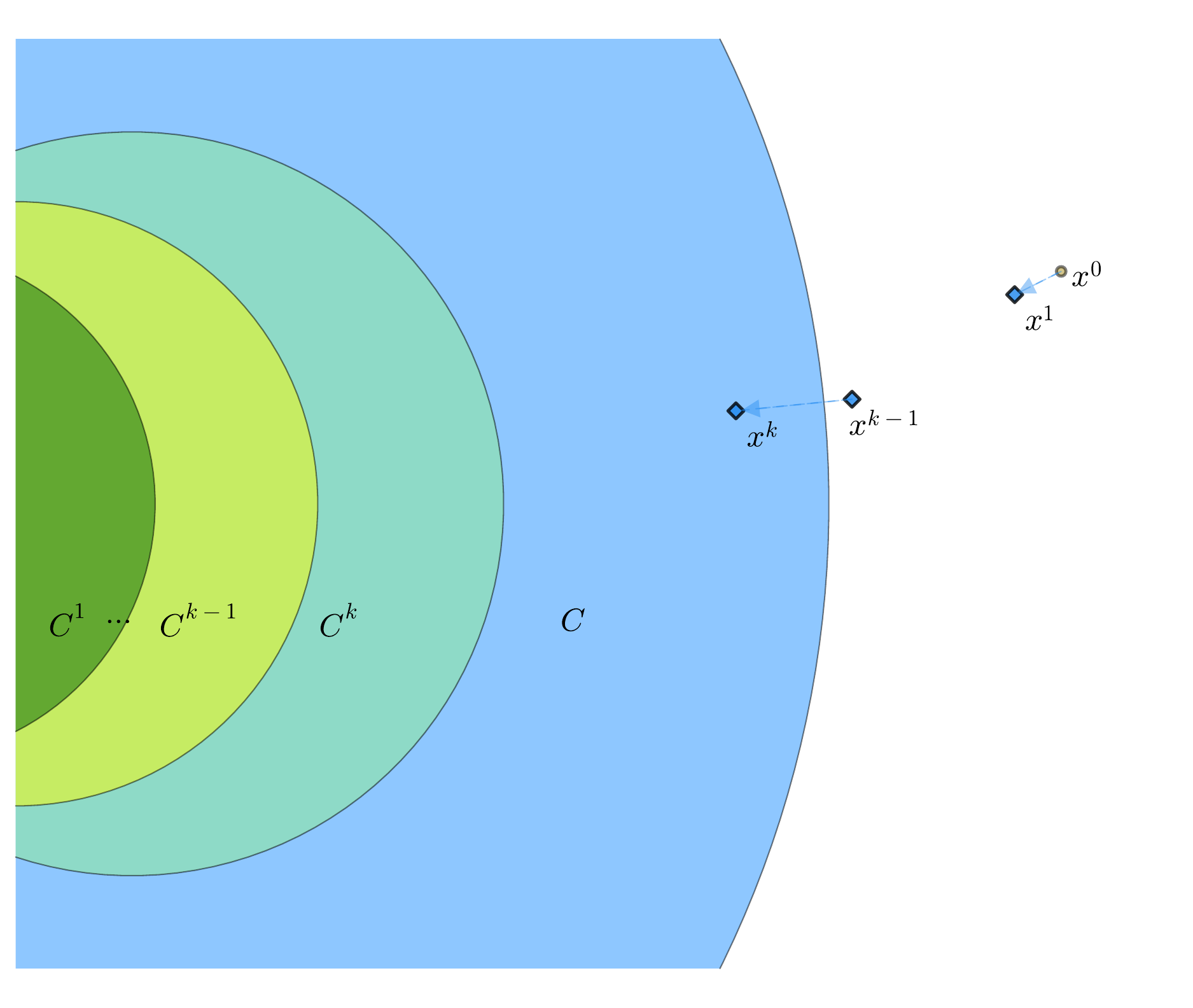}}
    \caption{Sketch of PACA finite convergence.}
    \label{pic.PACA}
   \end{figure}

We mention that, as far as we know, this is the first circumcenter-based method which is proved to be finitely convergent. The methods introduced in \cite{Behling:2024,Behling:2024b} achieve superlinear convergence assuming that the convex sets have smooth boundaries; our convergence analysis requires no smoothness hypothesis.

\begin{lemma}
\label{lem.epsilon_summable}
Let $\{x^k\}$ be the sequence  generated by PACA.  Under the Slater condition given in  \cref{def.Slater_condition}, if the sequence $\{x^k\}$ is infinite, then 
$\sum_{k=0}^\infty\epsilon_k<\infty$.
\begin{proof}
    Let $\delta_k\coloneqq \dist(x^k,C^k)$. By \cref{e33} in \Cref{lem.linearConvergence_distanceSequence.fake} and \cref{e43}, we have
   \[
        \label{eq.thm.finite_convergence.2}
        \delta_{k+1}\le\lambda\delta_k\le\dots\leq\lambda^{k+1-N} \delta_N,
    \]
    for all $k\geq N$.
    Let $z^k$ be the closest point to $x^k$ in $C^k$, so that $\delta_k=\lV x^k-z^k\rV$. Since $\{x^k\}$ is infinite, for all $k$ there exists $j(k)\in\{1,\ldots,m\}$ such that $f_{j(k)}(x^k) >0$. On the other hand, 
    $f_{j(k)}(z^k)+\epsilon_k\le 0$, because $z^k \in C^k$. Using the subgradient inequality, we have
    \[
    \label{eq.thm.finite_convergence.3}
        f_{j(k)}(z^k) \geq f_{j(k)}(x^k) +\left(u_{j(k)}^{k}\right)^\top (z^k-x^k).
    \]
Adding $\epsilon_k$ on both sides of \cref{eq.thm.finite_convergence.3}, we obtain
    \[
    \label{eq.thm.finite_convergence.4}
        0\ge f_{j(k)}(z^k)+\epsilon_k\ge f_{j(k)}(x^k) +\epsilon_k +\left(u_{j(k)}^{k}\right)^\top (z^k-x^k),
   \]
 Since $f_{j(k)}(x^k)>0$, we get from \cref{eq.thm.finite_convergence.4}
    \[
        \label{eq.thm.finite_convergence.4b}
        -\epsilon_k  \geq \left(u_{j(k)}^{k}\right)^\top (z^k-x^k).
    \]
    In view of the above inequality, it holds that 
    \[\label{e36}
    \epsilon_k\le\lv\left(u_{j(k)}^{k}\right)^\top (z^k-x^k)\rv
    \le\lV u_{j(k)}^{k}\rV\,\lV z^k-x^k\rV=\lV u_{j(k)}^{k}\rV\delta_k.
   \]
    From \cref{e36} and \cref{e26}, we have
    \[
        \label{eq.thm.finite_convergence.5}
        \delta_k\ge\frac{\epsilon_k}{\lV u_{j(k)}^{k}\rV}\ge \frac{\epsilon_k}{\sigma}.
    \]
    Now \cref{eq.thm.finite_convergence.2} and \cref{eq.thm.finite_convergence.5} imply 
    \[\label{e35}
     \epsilon_k\le\sigma\delta_N \lambda^{k+1-N},
     \] 
     for all $k\geq N$. Let $\zeta=\sum_{k=0}^{N-2}\epsilon_k$.
     From \cref{e35}, we obtain
    \[ 
    \sum_{k=0}^\infty\epsilon_k=\zeta+
    \sum_{k=N-1}^\infty\epsilon_k
    \le\zeta+\sigma\delta_N\sum_{k=N-1}^\infty\lambda^{k+1-N}=
    \zeta+\frac{\sigma\delta_N}{1-\lambda}<\infty,
    \]
    establishing the result.
\end{proof}
\end{lemma}

\begin{theorem}
    \label{thm.finite_convergence}
    Let $\{x^k\}$ be the sequence generated by PACA. Under the Slater condition given in \cref{def.Slater_condition}, if the sequence $\{\epsilon_k\}$ decreases to $0$ and $\sum_{k=0}^\infty\epsilon_k=\infty$, then PACA has finite termination,
    \emph{i.e.}, there exists some index $k$ such that $x^k$ solves CFP.
\end{theorem}
\begin{proof}
    Immediate from \Cref{lem.epsilon_summable}.
\end{proof}

We observe that there are many choices for the sequence $\{\epsilon_k\}$ which satisfy the assumptions of \Cref{thm.finite_convergence},
for instance, $\epsilon_k=\nu k^{-r}$, for any $\nu>0$ and any $r\in (0,1]$.

\section{Numerical experiments}\label{s7}

In order to investigate the behavior of PACA, we compare our proposed algorithm with two other approximate projections based algorithms: CRM on the Pierra's reformulation with approximate projection (denoted as \texttt{CARMprod}) (see \cite{Araujo:2022}),  the Simultaneous subgredient projection method (SSPM) of~\cite{Iusem:1986a}, and the Modified Cyclic subgradient projection (MCSP) by~\cite{DePierro:1988a}.  

Taking into account \cref{thm.finite_convergence}, we set the perturbation parameter for PACA defining $\epsilon_k \coloneqq \hat{\epsilon}(k) = \frac{1}{k}$ (algorithm \texttt{PACA1}) and $\epsilon_k \coloneqq   \bar{\epsilon}(k) = \frac{1}{\sqrt{k}} $ (algorithm \texttt{PACA2}). Therefore, we ensure that \texttt{PACA1} and \texttt{PACA2} have finite termination. 

We comment that SSPM is somehow similar to PACA: the difference relies on the way one set $\alpha_k$, instead of the one PACA uses. For PACA,  $\alpha_k$ is given in \cref{eq.def.alpha_k}, which arises from a circumcenter (see \cref{lem.PAC_and_CRM}). For SSPM, we set $\alpha_k = 1$ and for MCSP, $\alpha_k = \frac{1}{k}$. Meanwhile, MCSP just computes its iterates visiting all sets by means of  perturbed subgradients computed as in \cref{eq.def.v_i_k}.

Likewise to PACA, SSPM and MCSP also enforce a perturbation parameter, and thus \texttt{SSPM1} and \texttt{MCSP1} employ  $\epsilon_k \coloneqq \hat{\epsilon}(k)$, while \texttt{SSPM2} and \texttt{MCSP2} use $\epsilon_k \coloneqq \bar{\epsilon}(k)$.

These methods are applied to the problem of finding a point in the intersection of $m$ ellipsoids, \emph{i.e}, in \cref{eq.CFP_1}, each $C_i \coloneqq \xi_i$ is regarded as an ellipsoid  given by
\[
\xi_i\coloneqq\left\{x \in \re^n \mid f_i(x) \leq 0, \text { for } i=1,2, \ldots, m\right\},
\]
with $f_i: \re^n \rightarrow \re$ defined as \(
f_i(x)\coloneqq x^\top A_i x+2 x^\top b^i-c_i.\) For each $i=1, \ldots, m$, we have that  $A_i\in \re^{n\times n}$  is symmetric positive definite, $b^i$ is a vector, and $c_i$ is a positive scalar.

The ellipsoids are generated in accordance to~\cite[sect.~5]{Behling:2024}, so that the sets have not only nonempty intersection,  but also a Slater point. Moreover, the subdifferential of each $f_i$ is a singleton given by $\partial f_i(x) = 2(A_ix + b^i)$, since each $f_i$ is differentiable.

The computational experiments were performed on an Intel Xeon W-2133 3.60GHz with 32GB of RAM running Ubuntu 20.04 using Julia v1.9 \cite{Bezanson:2017}, and are fully available at \url{https://github.com/lrsantos11/CRM-CFP}.

\Cref{pic.1,tab.Ellipsoid_Time,tab.Ellipsoid_stats} summarize the results for different number of ellpsoids ($m$) and dimensions ($n$). \Cref{pic.1} is performance profile from Dolan and Moré~\cite{Dolan:2002}, a well known tool for comparing multiple algorithms on problem sets based on a performance metric and offers result visualization for benchmark experiments. The best algorighm is the one with higher percentage of problems solved within a given factor of the best time. 
In our case, we use wall-clock time (in seconds) as the performance measure.  
\Cref{tab.Ellipsoid_Time} shows the average wall-clock time per dimension and number of ellipsoids. \Cref{tab.Ellipsoid_stats} shows the statistics of all the experiments with Ellipsoids intersection considering wall-clock time. Our numerical results indicate  the PACA framework as the winner, with \texttt{PACA2} being the fastest algorithm.

Note that the numerical experiments also portray the finite convergence inherent to PACA, MCSP and SSPM, as  the solutions found by  all the algorithms \texttt{PACA1}, \texttt{PACA2}, \texttt{MCSP1}, \texttt{MCSP2}, \texttt{SSPM1} and \texttt{SSPM2} are interior.

\begin{figure}[htpb]
    \centering
 
    {\includegraphics[width=.8\textwidth]{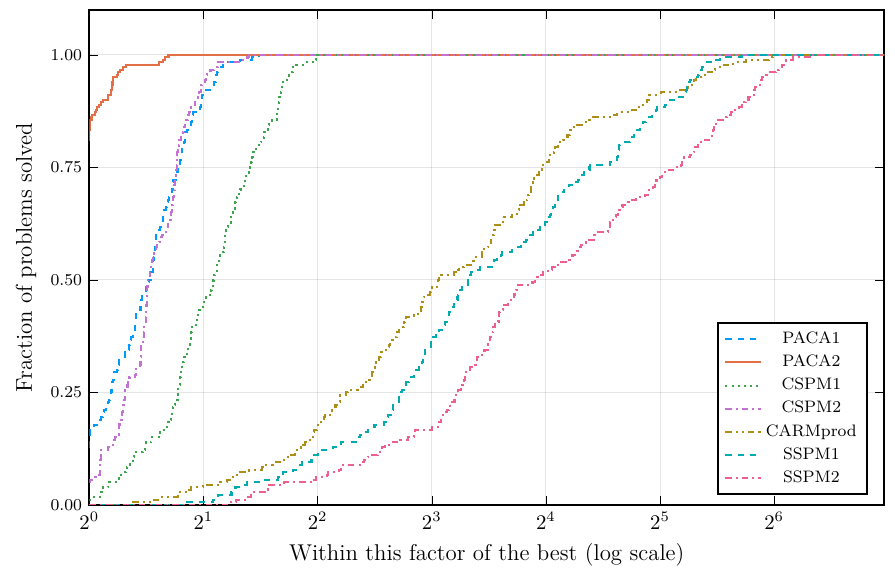}}
    \caption{Performance profile of experiments with Ellipsoids intersection considering wall-clock time (\unit{\second}).}
    \label{pic.1}
   \end{figure}

\sisetup{round-precision = 3, table-alignment=right, scientific-notation=fixed, fixed-exponent=-4, table-auto-round}
{\small \begin{table}[htpb]
    \caption{The average of wall-clock time ($\times\qty{e-4}{\second}$) per dimension ($n$) and number of Ellipsoids ($m$).}
\label{tab.Ellipsoid_Time}
\centering 
\begin{tabular}{rr
    S[table-format = 3.2]
    S[table-format = 3.2]
    S[table-format = 3.2]
    S[table-format = 3.2]
    S[table-format = 4.2]
    S[table-format = 4.2]
    S[table-format = 4.2]}
\toprule 
$n$  & $m$  &   \texttt{PACA1}  & \texttt{PACA2} & \texttt{CSPM1} & \texttt{CSPM2} & \texttt{CARMprod}  & \texttt{SSPM1} & \texttt{SSPM2}\\
\cmidrule(lr){3-9}

20 & 5 & 0.000174969 & 0.000130767 & 0.000229563 & 0.00028464 & 0.00100389 & 0.000657018 & 0.000802294 \\
  20 & 10 & 0.000149423 & 0.000116902 & 0.000249496 & 0.000182543 & 0.0017209 & 0.00114504 & 0.00173275 \\
  20 & 20 & 0.000451726 & 0.000329275 & 0.000744147 & 0.000517616 & 0.00618517 & 0.00496454 & 0.00786001 \\
  50 & 5 & 0.000177545 & 0.000180787 & 0.00025427 & 0.000276673 & 0.000791727 & 0.000738637 & 0.000929345 \\
  50 & 10 & 0.000358364 & 0.000304642 & 0.000561542 & 0.000467513 & 0.00210398 & 0.00236745 & 0.00345441 \\
  50 & 20 & 0.00105021 & 0.000717103 & 0.00137814 & 0.000984226 & 0.00902702 & 0.0106139 & 0.0161446 \\
  100 & 5 & 0.00206747 & 0.00202106 & 0.00269407 & 0.00217155 & 0.00554936 & 0.00865119 & 0.00948101 \\
  100 & 10 & 0.00464727 & 0.00373248 & 0.00450696 & 0.00356153 & 0.0228215 & 0.0245295 & 0.0316846 \\
  100 & 20 & 0.0339669 & 0.0278728 & 0.0283404 & 0.0363673 & 0.881335 & 0.344371 & 0.375855 \\
\bottomrule
\end{tabular}
\end{table}
}

\begin{table}[!ht]
    \caption{Statistics of all the experiments with Ellipsoids intersection considering wall-clock time ($\times\qty{e-4}{\second}$).}
    \label{tab.Ellipsoid_stats}

    \centering 
    \sisetup{round-precision = 3, table-alignment=right, scientific-notation=fixed, fixed-exponent=-4, table-auto-round}
\begin{tabular}{l
                S[table-format = 4.2]
                S[table-format = 2.2]
                S[table-format = 1.2]
                S[table-format = 6.2]}
    \toprule  
    & {\textbf{mean}} &    {\textbf{median}}   &  {\textbf{min}} &   {\textbf{max}}     \\
    \cmidrule(lr){2-5}
  \texttt{PACA1}     & 0.00478265  & 0.000380431 & 2.32151e-5 & 0.544721 \\
  \texttt{PACA2}     & 0.00393398  & 0.000294468 & 2.95695e-5 & 0.434008 \\
   \texttt{MCSPM1}   & 0.00432874  & 0.000611745 & 5.08408e-5 & 0.409626 \\
   \texttt{MCSPM2}   & 0.00497929  & 0.000405281 & 4.12874e-5 & 0.558935 \\
    \texttt{CARMprod} & 0.103393    & 0.00219481 & 0.000211971 & 16.0615 \\
   \texttt{MSSPM1}   & 0.0442265   & 0.00304216 & 0.00010617 & 5.32259 \\
   \texttt{MSSPM2}   & 0.0497716   & 0.00412424 & 0.000156793 & 5.41655 \\
    \bottomrule
    \end{tabular}
  \end{table}

\section{Concluding remarks}\label{s8}

In this work, we introduced a novel algorithm called Perturbed Approximate Circumcenter Algorithm (PACA) to address the Convex Feasibility Problem (CFP). 
The proposed algorithm ensures finite convergence under a Slater condition, making it unique in the landscape of circumcenter schemes for CFP.  In addition to yieding finite convergence, this method leverages projections onto perturbed halfspaces, which are explicitly computed in contrast to obtaining general  orthogonal projections onto convex set. Numerical experiments further showcase PACA's efficiency compared to existing methods.

\section*{Acknowledgments}

The authors are grateful to the anonymous referees and the handling editor for their valuable comments and suggestions which helped to improve the quality of the paper.

\bibliographystyle{siamplain}

\bibliography{refs}

\end{document}